\documentclass[12pt]{amsart}
\usepackage[usenames,dvipsnames]{xcolor}
\usepackage{amsmath,amssymb,amsfonts,amsthm,pgf,tikz, enumerate}
\usetikzlibrary{shapes,positioning,decorations}
\usetikzlibrary{decorations.pathreplacing}
\usepackage[bookmarks=true,pdfborder={0 0 0}]{hyperref}
\usepackage{soul}
\usepackage{graphicx}

\newtheorem{theorem}{Theorem}[section]
\newtheorem{lemma}[theorem]{Lemma}
\newtheorem{corollary}[theorem]{Corollary}

\newtheorem{proposition}[theorem]{Proposition}

\theoremstyle{definition}

\newtheorem{definition}[theorem]{Definition}

\theoremstyle{remark}

\numberwithin{equation}{section}

\DeclareMathOperator{\Dom}{Dom}

\usepackage{enumerate} 
\usepackage{bm}

\title[An Unbounded SISP for Infinite Graphs]{A Structured Inverse Spectrum Problem for Infinite Graphs and Unbounded Operators}
\author{Ehssan Khanmohammadi}
\address{Franklin \& Marshall College}
\email{ehssan@fandm.edu}

\begin{document}

\begin{abstract}
Given an infinite graph $G$ on countably many vertices, and a closed, infinite set $\Lambda$ of real numbers, we prove the existence of an unbounded self-adjoint operator whose graph is $G$ and whose spectrum is $\Lambda$. 

%Moreover, the set of limit points of $\Lambda$ equals the essential spectrum of $A$, and the isolated points of $\Lambda$ are eigenvalues of $A$ with multiplicity one. It is also shown that any two such matrices constructed by our method are approximately unitarily equivalent.
\end{abstract}

\maketitle   
\keywords{Keywords: Inverse Spectrum Problem, Infinite Graph, Unbounded Operator.%, Essential Spectrum, Approximate Unitary Equivalence.  Jacobian Method
\\

\subjclass{AMS MSC: 05C50,  05C63, 15A29, 47A10.}
%05C50 Graphs and linear algebra (matrices, eigenvalues, etc.)
%05C63 Infinite graphs
%15A29 Inverse problems
%47A10 Spectrum, resolvent

\section{Introduction}
The main theorem in 
%a previous work of ours 
\cite{hk16} states that given an infinite graph $G$ on countably many vertices, and a compact, infinite set $\Lambda$ of real numbers there is a real symmetric matrix whose graph is $G$ and whose spectrum is $\Lambda$. More precisely, one can construct a bounded self-adjoint operator $T$ on $\ell^2$ with spectrum $\Lambda$ such that the matrix of $T$ with respect to the standard basis of $\ell^2$ has the desired zero-nonzero pattern given by the graph $G$. Here $\ell^2$ is the short form for the Hilbert space of square-summable real sequences $\ell^2(\mathbb{N})$.

In this article we relax the compactness condition on $\Lambda$ in the result above by working with unbounded operators. Note that the spectrum of any bounded operator is a compact subset of the complex plane, so the main result of \cite{hk16} is in this sense optimal for bounded operators. Additionally, since the spectrum of any unbounded operator is a closed subset of the complex plane (see, for instance, \cite[Proposition 2.6]{schmudgen12}), a proper generalization of this result to the case of unbounded operators should replace the compactness assumption of $\Lambda$ by closedness. This is precisely what we accomplish in this paper.

Throughout, all vector spaces will be over the field of real numbers making inner products $\langle v, w\rangle$ linear in both $v$ and $w$. We denote the operator norm by $\|\cdot \|_{\text{op}}$.
\section{Preliminaries}
In this section we recall some definitions and establish a few basic results that we shall use later.
\begin{definition}
An \emph{unbounded} operator $T$ on a Hilbert space $\mathcal H$ is a linear map of some dense subspace $\Dom(T)\subset \mathcal H$ into $\mathcal H$.
\end{definition}
According to this definition, `unbounded' means `not necessarily bounded,' in the sense that we allow $\Dom(T)=\mathcal H$ if $T$ is bounded.

\begin{definition} Suppose $T$ is an unbounded operator on $\mathcal H$. Let $\Dom(T^*)$ be the space of all $v\in \mathcal H$ for which the linear functional 
	\[
	v\mapsto \langle v, Tw\rangle, \quad w\in \Dom(T),
	\]
is bounded. For $v\in \Dom(T^*)$, we define $T^*v$ to be the unique vector such that $\langle T^* v, w\rangle=\langle v, Tw\rangle$ for all $w\in \Dom(T)$.
\end{definition}

\begin{definition} An unbounded operator $T$ on $\mathcal H$ is
\begin{enumerate}[(1)]
\item \emph{symmetric} if $\langle v, Tw\rangle=\langle Tv, w\rangle$ for all $v, w\in \Dom(T)$, and in particular
\item \emph{self-adjoint} if $\Dom(T)=\Dom(T^*)$ and $T^*v=Tv$ for all $v$ in $\Dom(T)$.
\end{enumerate} 
It is easy to check that $T$ is symmetric if and only if $T^*$ is an \emph{extension} of $T$, that is, $\Dom(T)\subset\Dom(T^*)$ and $T=T^*$ on $\Dom(T)$.
\end{definition}

The following proposition, involving a `discrete version' of the potential energy operator in quantum mechanics, will play a key role in proving our main result. Indeed, the spectral theorem implies that this multiplication operator is the prototype of all self-adjoint operators. See, for instance, Chapters 9 and 10 in \cite{hall13}.
\begin{proposition}\label{P:potential operator is selfadjoint} Let $\{\lambda_n\}_{n=1}^\infty$ be any sequence of real numbers. Let $T$ be the unbounded operator on $\ell^2$ with domain
	\[
	\Dom(T)=
	\{
	\{a_n\}_{n=1}^\infty\in \ell^2\mid \{\lambda_n a_n\}_{n=1}^\infty\in \ell^2
	\}
	\]
such that $T$ maps $\{a_n\}_{n=1}^\infty\in \Dom(T)$ to $\{\lambda_na_n\}_{n=1}^\infty$. Then $T$ is self-adjoint.
\end{proposition}
\begin{proof} First, observe that $\Dom(T)$ contains all finite sequences and hence it is dense in $\ell^2$. Next, since $\{\lambda_n\}_{n=1}^\infty$ is a sequence of real numbers, $T$ is clearly symmetric and thus $T^*$ is an extension of $T$. It remains to show that $\Dom(T^*)\subset\Dom(T)$.

Suppose $\mathbf{b}\in \Dom(T^*)$ so that
	\[
	\mathbf{a}\mapsto \langle \mathbf{b}, T\mathbf{a}\rangle= \sum_{n=1}^\infty b_n\lambda_na_n,\quad \mathbf{a}\in \Dom(T),
	\]
is a bounded functional. This functional has a unique bounded extension to $\ell^2$ and therefore, by the Riesz representation theorem, it can be represented by a unique $\mathbf{c}\in \ell^2$. Thus,
	\[
	\sum_{n=1}^\infty b_n\lambda_na_n=\sum_{n=1}^\infty c_na_n
	\]
or 
	\[
	\sum_{n=1}^\infty (b_n\lambda_n-c_n)a_n=0
	\]
for all $\mathbf{a}\in \ell^2$. This immediately implies $b_n\lambda_n=c_n$ for all $n$ and hence $\mathbf{b}\in \Dom(T)$, yielding $\Dom(T^*)\subset\Dom(T)$.
\end{proof}
Let us recall the definition of the spectrum of an unbounded operator.
\begin{definition} Let $T$ be an unbounded operator on $\mathcal H$. A number $\lambda\in \mathbb{C}$ is in the \emph{resolvent set} of $T$ if there exists a bounded operator $S$ with the following properties: for all $v\in \mathcal H$, $Sv$ belongs to $\Dom(T)$ and $(T-\lambda I)S v=v$, and for all $w\in \Dom(T)$, $S(T-\lambda I) w=w$. 

The complement of the resolvent set of $T$ is called the \emph{spectrum} of $T$ and is denoted by $\sigma(T)$.
\end{definition}
For instance, one can easily check that the spectrum of the multiplication operator $T$ in Proposition~\ref{P:potential operator is selfadjoint} is the closure of $\{\lambda_n\mid n\in \mathbb{N}\}$ as a subset of the real line.
\begin{definition} A sequence $\{T_n\}_{n=1}^\infty$ of unbounded operators on a Hilbert space $\mathcal{H}$ is said to be \emph{convergent} to an unbounded operator $T$ if for each sufficiently large $n$, $T-T_n$ is bounded on $\Dom(T_n)\cap \Dom(T)$ and moreover $\|T-T_n\|_{\text{op}}\to 0$ as $n\to \infty$.
\end{definition}
\begin{lemma} \label{L:limit of self-adjoints}
Suppose $\{T_n\}_{n=1}^\infty$ is a sequence of self-adjoint operators that is convergent to an unbounded operator $T$ on a Hilbert space $\mathcal{H}$. Assume that $\Dom(T_n)=\mathcal{D}$ for all $n$, where $\mathcal{D}$ is some dense subspace of $\mathcal{H}$. Then $T$ is self-adjoint on $\mathcal{D}$.
\end{lemma}
\begin{proof}
Clearly $T$ is symmetric, because each $T_n$ is symmetric on $\mathcal{D}$ and hence for all $v, w\in \mathcal{D}$,
	\[
	\langle w, T v\rangle =\lim_{n\to \infty} \langle w, T_n v\rangle=\lim_{n\to \infty} \langle T_n w, v\rangle= \langle Tw, v\rangle.
	\]
Thus, $T^*$ is an extension of $T$ and $\Dom(T^*)\supset \Dom(T)=\mathcal{D}$. 

Now let $w\in \Dom(T^*)$ so that $v\mapsto \langle w, Tv\rangle$ is bounded for $v\in \mathcal{D}$. We claim that $w\in \mathcal{D}$. This is clear if $v\mapsto \langle w, T_n v\rangle$ is bounded on $\mathcal{D}$ for some $n$, because in that case $w\in \Dom(T_n^*)=\Dom(T_n)=\mathcal{D}$. So, we assume that there exists a sequence of unit vectors $\{v_n\}_{n=1}^\infty$ in $\mathcal{D}$ such that $|\langle w, T_n v_n\rangle|>n$ for each $n$. Thus,
	\begin{align*}
	\Bigl||\langle w, T v_n\rangle|-|\langle w, T_n v_n\rangle|\Bigr|
	&\le
	|\langle w, T v_n\rangle-\langle w, T_n v_n\rangle|
	\\
	&\le
	\|w\| \|T-T_n\|_{\text{op}},
	\end{align*}
by an application of the reverse triangle inequality and the Cauchy--Schwarz inequality. The right side of the second inequality above tends to $0$ as $n$ goes to $\infty$ implying that $|\langle w, T v_n\rangle|\to \infty$. This is absurd; hence $w\in \mathcal{D}$. Therefore, $\Dom(T^*)\subset \mathcal{D}$ which finishes the proof that $T$ is self-adjoint on $\mathcal{D}$.
\end{proof}
Finally, to finish this section, we record a lemma whose easy proof we omit.
\begin{lemma}\label{L:sum of self-adjoints} Let $\mathcal{H}$ be a Hilbert space with an orthonormal basis $\mathfrak{B}$. Suppose $\mathfrak{B}_1$ and $\mathfrak{B}_2$ are two subsets of $\mathfrak{B}$ that partition $\mathfrak{B}$ and denote the Hilbert spaces generated by them by $\mathcal{H}_1$ and $\mathcal{H}_2$, respectively. If $T_i$ are unbounded self-adjoint operators on $\mathcal{H}_i$, for $i=1, 2$, then the operator $T$ defined by $T_1\oplus T_2$ is an unbounded self-adjoint operator on $\mathcal{H}$ with $\Dom(T)=\Dom(T_1)\oplus \Dom(T_2)$.
\end{lemma}
\section{Main Theorem}
In preparation for our main result, now we introduce the notion of a graph of a symmetric matrix (or a self-adjoint operator). 
\begin{definition} Let $G$ be a (finite or infinite) graph whose vertices are indexed by $1, 2, \dots$. We say that $G$ is the \emph{graph of a real symmetric matrix} $A=[a_{ij}]$ if for any $i\ne j$, we have $a_{ij}\ne 0$ precisely when the vertices $i$ and $j$ are adjacent in $G$.

We say that $G$ is the \emph{graph of a self-adjoint operator} $T$ on $\ell^2$ if $G$ is the graph of the standard matrix of $T$.
\end{definition}
The following theorem is proved in \cite{hk16} using the so-called Jacobian method, and the interested reader may want to consult that article for the details of the Jacobian method and relevant references.
\begin{theorem} \label{finitelambdasiep}
Let $\{ \lambda_n\}_{n=1}^\infty$ be a sequence of distinct real numbers and suppose $\{G_n\}_{n=1}^\infty$ is a sequence such that for each $n\in \mathbb{N}$, $G_n$ is a graph on $n$ vertices and also a subgraph of $G_{n+1}$. Then for any sequence of positive numbers $\{\varepsilon_n\}_{n=1}^\infty$ we can find a sequence of symmetric matrices $\{A_n\}_{n=1}^\infty$ such that for any $n\in \mathbb{N}$,
\begin{enumerate}[(i)]
	\item $A_n$ has graph $G_n$ and spectrum $\{\lambda_1, \dots, \lambda_n\}$, and 
	\item 
	$
	\|A_{n}\oplus [\lambda_{n+1}]-A_{n+1}\|_{\textrm{op}}<\varepsilon_n.
	$\label{I: closeness condition}
\end{enumerate}
\end{theorem}
The next theorem can be thought of as a `continuity property' of the spectrum of self-adjoint operators.
\begin{theorem}\label{T:kato perturbation}
Let $T$ and $A$ denote a self-adjoint and a bounded symmetric operator on a Hilbert space $\mathcal H$, respectively.  Then $S=T+A$ is self-adjoint and the Hausdorff distance between the spectra of $S$ and $T$, namely $d(\sigma(S), \sigma(T))$, satisfies
	\[
	d(\sigma(S), \sigma(T))\le \|A\|_{\textrm{op}}.
	\]
\end{theorem}
This theorem, whose proof can be found in \cite[Theorem 4.10]{kato95}, immediately implies the following corollaries that we shall need later on.
\begin{corollary}\label{C:Closeness of spectra}
Let $\{T_n\}_{n=1}^\infty$ be a sequence of unbounded self-adjoint operators on a Hilbert space $\mathcal{H}$. Assume that $\{T_n\}_{n=1}^\infty$ converges to a self-adjoint operator $T$ and that $\Dom(T)\cap \Dom(T_n)$ is dense in $\mathcal{H}$ for all $n$. Then for any $\lambda\in \sigma(T)$ and any neighborhood $U$ of $\lambda$, there exists an $N\in \mathbb{N}$ such that $U$ intersects $\sigma(T_n)$ nontrivially for all $n>N$.
\end{corollary}
\begin{proof}
Since $\Dom(T-T_n)=\Dom(T)\cap \Dom(T_n)$, the density of the right side in $\mathcal{H}$ guarantees that the difference $T-T_n$ of self-adjoint operators is symmetric. Also, $\|T-T_n\|_{\text{op}}\to 0$ implies that, for sufficiently large $n$, $T-T_n$ is bounded on $\Dom(T)\cap \Dom(T_n)$ and hence it can be extended to a bounded symmetric operator on $\mathcal{H}$. By definition of the Hausdorff distance, 
	\[
	d(\sigma(T-T_n), \{\lambda\})\le d(\sigma(T-T_n), \sigma(T))
	\]
for $\lambda\in \sigma(T)$. Now the corollary follows from Theorem~\ref{T:kato perturbation}.
\end{proof}%In particular if all $\sigma(T_n)=\Lambda$ where $\Lambda$ is a closed set, then $\lambda\in \Lambda$. This is what we need in the main theorem.
If $\{T_n\}_{n=1}^\infty$ is a sequence of noninvertible bounded operators on a Hilbert space and $\{T_n\}_{n=1}^\infty$ converges to an operator $T$, then $T$ is also noninvertible. This is a well-known consequence of the openness of the invertibility condition in unital Banach algebras. Instead of explicitly referring to noninvertibility of $T_n$, one can equivalently assume that $0$ belongs to $\sigma(T_n)$. This formulation has the advantage of making sense in more general contexts such as the next corollary.
\begin{corollary}\label{C:transfer of spectral values in the limit} 
Suppose $\{T_n\}_{n=1}^\infty$ is a sequence of self-adjoint operators on a Hilbert space $\mathcal{H}$ with $\Dom(T_n)=\mathcal{D}$, for $n=1, 2, \dots$, where $\mathcal{D}$ is a dense subspace of $\mathcal{H}$. If $\{T_n\}_{n=1}^\infty$ is convergent to an operator $T$ and $\lambda\in \sigma(T_n)$ for all $n$, then $\lambda\in \sigma(T)$.
\end{corollary}
\begin{proof}
Observe that $T$ is a self-adjoint operator on $\mathcal{D}$ by Lemma~\ref{L:limit of self-adjoints}. Thus, $T-T_n$ is bounded and symmetric on $\mathcal{D}$ for each $n$. If $\lambda\not\in \sigma(T)$, then, since $\sigma(T)$ is a closed subset of $\mathbb{R}$, there exists an open subset $U$ of $\mathbb{R}$ containing $\lambda$ that is disjoint from $\sigma(T)$; hence, $0<d(\{\lambda\}, \sigma(T))$. This, together with~Theorem~\ref{T:kato perturbation}, implies that, for each $n$,
	\[
	0<d(\sigma(T_n), \sigma(T))\le\|T-T_n\|_{\text{op}}
	\]
which is in contradiction with the assumption that $\{T_n\}_{n=1}^\infty$ is convergent to $T$. Therefore, $\lambda$ must be in $\sigma(T)$.
\end{proof}
We are ready to state and prove our main theorem. This is done by taking the limit, in a suitable sense, of the matrices that are constructed as in Theorem \ref{finitelambdasiep}.

\begin{theorem}%[SISP with data $(G, \Lambda)$]
\label{T:infinite bounded lambdasisp}
	Given an infinite graph $G$ on countably many vertices and a closed, infinite set $\Lambda$ of real numbers, there exists an unbounded self-adjoint operator $T$ on the Hilbert space $\ell^2$ such that 
	\begin{enumerate}[(i)]
	\item the (approximate point) spectrum of $T$ equals $\Lambda$, and \label{I: sigma_ap=Lambda}
	\item the (real symmetric) standard matrix of $T$ has graph $G$.
	\end{enumerate}
\end{theorem}
\begin{proof} Let $\{\lambda_1, \lambda_2, \dots\}$ denote a countable dense subset of $\Lambda$. Suppose the vertices of $G$ are labeled by the numbers in $\mathbb{N}$ and for each $n\in \mathbb{N}$ let $G_n$ be the induced subgraph of $G$ on the first $n$ vertices. By Theorem~\ref{finitelambdasiep}, for any $\varepsilon>0$ we can find matrices $\{A_{n}\}_{n=1}^\infty$ such that $A_n$ has graph $G_n$ and spectrum $\{\lambda_1, \dots, \lambda_n\}$, and moreover,
	\begin{equation}\label{E:norm-difference}
	\|
		A_{n}\oplus[\lambda_{n+1}]
		-A_{n+1}
	\|_{\textrm{op}}
	<
	\frac{\varepsilon}{2^n}.
	\end{equation}
For each $n$ define the unbounded linear operator $T_n$ on the Hilbert space of square-summable sequences $\ell^2$ with domain
	\[
	\mathcal{D}=
	\{
	\{a_n\}_{n=1}^\infty\in \ell^2\mid \{\lambda_na_n\}_{n=1}^\infty\in \ell^2
	\}
	\]		
	 such that 
	 \[
	 M_n=A_n\oplus \operatorname{diag}(\lambda_{n+1}, \lambda_{n+2}, \dots)
	 \]
is the matrix representation of $T_n$ with respect to the standard Hilbert basis $\mathfrak{B}=\{\bm {e_1}, \bm {e_2}, \dots\}$ of $\ell^2$. (Note that the definition of $\Dom(T_n)$ does not depend on the value of $n$.) Proposition~\ref{P:potential operator is selfadjoint} and Lemma~\ref{L:sum of self-adjoints} imply that $T_n$ is self-adjoint. It follows from \eqref{E:norm-difference} that for any $i$ in $\mathbb{N}$ we have
	\[
	\|M_n \bm{e_i}-M_{n+1}\bm {e_i}\|_2<\frac{\varepsilon}{2^n}.
	\]
Thus, the sequence of partial sums $\{\sum_{k=1}^{n-1} (M_{k+1}\bm {e_i}-M_k\bm {e_i})\}_{n=1}^\infty$ is absolutely convergent, and therefore the sequence $\{M_n\bm {e_i}\}_{n=1}^\infty$ satisfying
	\[
	M_{n} \bm {e_i} = M_1\bm {e_i} +\sum_{k=1}^{n-1} (M_{k+1}\bm {e_i}-M_k\bm {e_i})
	\]
is convergent in $\ell^2$. Let $M$ denote the matrix whose columns are obtained by this limiting process, that is, $M$ is the matrix that $M\bm{e_i}=\lim_{n\to \infty} M_n \bm{e_i}$ for each $i\in \mathbb{N}$. Note that for each $n = 1,2,\dots$ the graph of $A_n$ is the induced subgraph of $G$ on the first $n$ vertices. Thus, by construction, $G$ is the graph of $M$. Our next objective is showing that $M$ is indeed the standard matrix of an unbounded linear operator $T$ on $\ell^2$. 
Observe that $T_m-T_n$ is a bounded operator on $\mathcal{D}$, by construction of $T_m$ and $T_n$; therefore, $T_m-T_n$ has a unique bounded extension to $\mathcal{B}(\ell^2)$. We shall denote this extension by $T_m-T_n$ as well.
\begin{align*}
\|T_n - T_{n+1}\|_{\textrm{op}} 
%&= \sup_{\|\bm v\|_2=1} \{ \|(T_n - T_{n+1}) \bm v\|_2\}
%\\
&=
\sup_{\|\bm v\|_2=1} 
\|T_n \bm v - T_{n+1} \bm v\|_2\\
&= \sup_{\|\bm v\|_2=1}
\left\| 
 \left[ \begin{array}{c}
\left[ 
	\begin{array}{c}
	A_n
	\left[ 
	\begin{array}{c}
	v_1 \\ \vdots \\ v_n
	\end{array} 
\right] 
\\ 
\hline
\lambda_{n+1} v_{n+1}
\end{array} \right] - A_{n+1} \left[ \begin{array}{c}
v_1 \\ \vdots \\ v_{n+1}
\end{array} \right]\\ \hline
\begin{array}{c}
0 \\ \vdots \\ 0
\end{array}
\end{array} 
\right] 
\right\|_2  
\\
&= 
\sup_{\|\bm v\|_2=1} 
	\left\|
	\left( \left[ \begin{array}{c|c}
A_n & \\ \hline
& \lambda_{n+1}
\end{array} \right] - A_{n+1}
	\right) 
\left[ 
\begin{array}{c}
v_1 \\ \vdots \\ v_{n+1}
\end{array} 
\right] 
\right\|_2
\\
&< \frac{\varepsilon}{2^n},
\end{align*}
where the inequality in the last line is due to the submultiplicative property of the operator norm together with \eqref{E:norm-difference}. This inequality immediately implies that the sequence of partial sums $\{\sum_{k=1}^{n-1} (T_{k+1}-T_k)\}_{n=1}^\infty$ is absolutely convergent in the Banach space of bounded operators $\mathcal{B}(\ell^2)$, and hence the sequence $\{T_n\}_{n=1}^\infty$ satisfying
	\[
	T_n=T_1 +\sum_{k=1}^{n-1} (T_{k+1}-T_k)
	\]
is convergent to an unbounded operator $T$. This means that we define $T$ by
	\[
	T=T_1+ \lim_{n\to \infty} \sum_{k=1}^{n-1} (T_{k+1}-T_k),
	\]
which is the sum of the unbounded self-adjoint operator $T_1$ and a bounded self-adjoint operator on $\ell^2$. Therefore, $T$ is self-adjoint with domain given by $\mathcal{D}=\Dom(T_1)$.
Since for each $i\in \mathbb{N}$ we have $T \bm e_i =  \lim_{n\to \infty} T_n \bm e_i$ and $T_n \bm e_i=M_n\bm e_i$, we conclude that $T\bm e_i=M\bm e_i$ and thus $M$ is the standard matrix of $T$. 

It remains to prove that $\sigma(T)=\Lambda$.  First, we claim that each $\lambda_i \in \{\lambda_1, \lambda_2, \dots \}\subset\Lambda$ is in the spectrum of $T$. To see this, note that $T_n$ was defined so that $\{\lambda_1, \lambda_2, \dots\}\subset \sigma(T_n)$ for each $n$. Hence, by Corollary~\ref{C:transfer of spectral values in the limit} we have $\{\lambda_1, \lambda_2, \dots\}\subset \sigma(T)$, as claimed. By taking closures, this inclusion implies $\Lambda \subset \sigma(T)$, because $\{\lambda_1, \lambda_2, \dots\}$ is dense in $\Lambda$ and $\sigma(T)$ is closed in $\mathbb{R}$.

Next, since the sequence $\{T_n\}_{n=1}^\infty$ is convergent to $T$ and $\sigma(T_n)=\Lambda$ for all $n$, by Corollary~\ref{C:Closeness of spectra} we conclude that for any $\lambda\in \sigma(T)$, every neighborhood of $\lambda$ intersects the closed set $\Lambda$. Hence, the reverse inclusion $\sigma(T)\subset \Lambda$ is also established. 

Finally, to complete the proof of Point \eqref{I: sigma_ap=Lambda} in the statement of the theorem note that the spectrum of any self-adjoint operator equals its approximate point spectrum, and, as shown above, $T$ is self-adjoint.
\end{proof}
%\section*{Acknowledgement}
%The author wishes to thank Keivan Hassani Monfared and Omid Khanmohamadi for some valuable suggestions for improving an earlier draft of this article.
%\vfill
\bibliographystyle{plain}
\bibliography{ref151217}

\end{document}